\documentclass[oneside,12pt]{amsart}
\usepackage{eurosym}
\usepackage{amsfonts}
\usepackage{amsmath,amssymb,amsthm,color}
\usepackage{amsopn}
\usepackage{latexsym}
\usepackage{float}
\usepackage{bm}
\usepackage[utf8]{inputenc}
\usepackage{hhline}
\usepackage[hidelinks]{hyperref}
\usepackage{calc}
\newsavebox\CBox
\newcommand\hcancel[2][0.5pt]{%
  \ifmmode\sbox\CBox{$#2$}\else\sbox\CBox{#2}\fi%
  \makebox[0pt][l]{\usebox\CBox}%
  \rule[0.5\ht\CBox-#1/2]{\wd\CBox}{#1}}
\usepackage{blindtext}
\usepackage{color}
\usepackage{float}
\usepackage{enumerate}
\usepackage{enumitem}   
\usepackage{bbm}
\usepackage{cancel} 
\usepackage{mathrsfs}
\usepackage{colonequals}
\usepackage{pdfpages} 
\usepackage{amsfonts,amsthm,bm}
\usepackage{xcolor}
\usepackage{mathtools}
\usepackage{booktabs}
\usepackage[normalem]{ulem}  
\numberwithin{equation}{section}
\usepackage[toc,page]{appendix} 
\usepackage{tikz-cd}
\theoremstyle{definition}  
\usepackage{stackengine}
\usepackage{latexsym}
\usepackage[font=small]{caption}
\usepackage[all]{hypcap} 
\usepackage[utf8]{inputenc}
\usepackage{physics}
\usepackage{framed}
\usepackage{pgf,tikz,pgfplots}
\usepgfplotslibrary{external} 
\AtBeginEnvironment{tikzcd}{\tikzexternaldisable}
\AtEndEnvironment{tikzcd}{\tikzexternalenable}
\usepackage{verbatim}
\usepackage{placeins}
\usepackage{multirow}
\usepackage{tabularx}
\usepackage{parskip} 

\title[Generalised Spin$^r$ Structures on Homogeneous Spaces]{Generalised Spin$^r$ Structures \\ on Homogeneous Spaces} 
\author[Diego Artacho and Marie-Amélie Lawn]
{Diego Artacho and Marie-Amélie Lawn} 

\address{Department of Mathematics, Imperial College London, London, SW7 2AZ, United Kingdom.}
\email{ d.artacho21@imperial.ac.uk,m.lawn@imperial.ac.uk}

\allowdisplaybreaks

\def\@makechapterhead#1{%
  \vspace*{10\p@}%
  {\parindent \z@ \raggedright \sffamily
    \interlinepenalty\@M
    \Huge\bfseries \thechapter \space\space #1\par\nobreak
    \vskip 30\p@
  }}

\def\@makeschapterhead#1{%
  \vspace*{10\p@}%
  {\parindent \z@ \raggedright
    \sffamily
    \interlinepenalty\@M
    \Huge \bfseries  #1\par\nobreak
    \vskip 30\p@
  }}

\makeatletter
\DeclareRobustCommand\mapstofill{%
  $\m@th
  {\mapstochar}%
  \smash-\mkern-7mu
  \cleaders\hbox{$\mkern-2mu\smash-\mkern-2mu$}\hfill
  \mkern-7mu
  \mathord\rightarrow
  $%
}
\makeatother

\makeatletter
\def\subsection{\@startsection{subsection}{2}%
  \z@{.5\linespacing\@plus.7\linespacing}
{.5\baselineskip}%
  {\normalfont\flushleft\scshape}%
}
\makeatother

\usepackage{etoolbox}
\patchcmd{\section}{\scshape}{\bfseries}{}{}
\patchcmd{\subsection}{\scshape}{\bfseries}{}{}
\patchcmd{\subsubsection}{\scshape}{\bfseries}{}{}
\makeatletter
\renewcommand{\@secnumfont}{\bfseries}
\makeatother

\makeatletter
\def\@tocline#1#2#3#4#5#6#7{\relax
  \ifnum #1>\c@tocdepth 
  \else
    \par \addpenalty\@secpenalty\addvspace{#2}%
    \begingroup \hyphenpenalty\@M
    \@ifempty{#4}{%
      \@tempdima\csname r@tocindent\number#1\endcsname\relax
    }{%
      \@tempdima#4\relax
    }%
    \parindent\z@ \leftskip#3\relax \advance\leftskip\@tempdima\relax
    \rightskip\@pnumwidth plus4em \parfillskip-\@pnumwidth
    #5\leavevmode\hskip-\@tempdima
      \ifcase #1
       \or\or \hskip 1em \or \hskip 2em \else \hskip 3em \fi%
      #6\nobreak\relax
    \hfill\hbox to\@pnumwidth{\@tocpagenum{#7}}\par
    \nobreak
    \endgroup
  \fi}
\makeatother

\newcommand{\restr}[1]{|_{#1}}

\newcommand{\mykeywords}[1]{\textbf{\emph{Key words: }} #1}
\newcommand{\msc}[1]{\textbf{\textbf{MSC: }} #1}

\DeclareMathOperator{\pr}{pr}

\DeclareMathOperator{\Ort}{O}
\DeclareMathOperator{\SO}{SO}
\DeclareMathOperator{\Spin}{Spin}
\DeclareMathOperator{\SU}{SU}
\DeclareMathOperator{\Sp}{Sp}
\DeclareMathOperator{\U}{U}

\theoremstyle{definition}
\newtheorem{deff}{Definition}[section]

\newtheorem{rem}[deff]{Remark}
\newtheorem{ex}[deff]{Example}

\theoremstyle{definition}
\newtheorem{prop}[deff]{Proposition}
\newtheorem{thmm}[deff]{Theorem}
\newtheorem*{thmm*}{Theorem}
\newtheorem{lemma}[deff]{Lemma}
\newtheorem{cor}[deff]{Corollary}
\newtheoremstyle{named}{}{}{\itshape}{}{\bfseries}{.}{.5em}{\thmnote{#3's }#1}
\theoremstyle{named}

\begin{document}


\maketitle  

\begin{abstract} 
  Spinorial methods have proven to be a powerful tool to study geometric properties of spin manifolds. Our aim is to continue the spinorial study of manifolds that are not necessarily spin. We introduce and study the notion of $G$-invariance of spin$^r$ structures on a manifold $M$ equipped with an action of a Lie group $G$. For the case when $M$ is a homogeneous $G$-space, we prove a classification result of these invariant structures in terms of the isotropy representation. As an example, we study the invariant spin$^r$ structures for all the homogeneous realisations of the spheres. 
\end{abstract}

\mykeywords{Spin geometry, homogeneous spaces, holonomy, generalised spin structures, spin$^c$, spin$^h$. }

\msc{53C27; 57R15; 53C30. }
 
\tableofcontents

\section{Introduction}\label{section:intro} 
The field of spin geometry has proven to be a fruitful area of study with numerous applications in various areas of mathematics. It has been applied to a wide range of problems, from the study of topology and differential equations to the theory of quantum mechanics. However, by their nature, spinorial considerations are limited to manifolds that admit a spin structure. To overcome this limitation, considerable efforts have been made in recent decades to extend spin geometry to non-spin manifolds. Recent developments have led to the introduction of spin$^{\mathbb{C}}$ and spin$^{\mathbb{H}}$ structures, which come, respectively, from the complexification and quaternionification of the classical spin groups~\cite{friedrich,moroianu,bar_spinh,spinq}. 

The idea of extending spin geometry to non-spin manifolds was introduced by Friedrich and Trautmann~\cite{FT}, who started the theory of spinorial Lipschitz structures. These ideas have been developed further in~\cite{CS1,CS2,CS3}. 

While every spin manifold is spin$^{\mathbb{C}}$ and every spin$^{\mathbb{C}}$ manifold is in turn spin$^{\mathbb{H}}$, the converses are not true. For example, even-dimensional complex projective spaces $\mathbb{CP}^{2n}$ are not spin, but they are spin$^{\mathbb{C}}$. Similarly, the Wu manifold $W = \SU(3)/\SO(3)$ is not spin$^{\mathbb{C}}$, yet it is spin$^{\mathbb{H}}$~\cite{thesis}. Moreover, there exist examples of manifolds which are not even spin$^{\mathbb{H}}$, such as $W \times W$~\cite{AM21}.  

It is a well-known fact that every almost-complex Riemannian manifold is spin$^{\mathbb{C}}$, and every almost-quaternionic manifold is spin$^{\mathbb{H}}$. Furthermore, it is known that every oriented Riemannian manifold of dimension $\leq 3$ is spin. In 2005, Teichner and Vogt~\cite{Teichner} completed the work started by Hirzebruch and Hopf in 1958~\cite{hopf}, showing that every oriented Riemannian manifold of dimension $\leq 4$ admits a spin$^{\mathbb{C}}$ structure. These structures are crucial for Seiberg-Witten theory, which has become an essential tool in the study of smooth $4$-manifolds~\cite{Morgan}. In 2021, Albanese and Milivojević~\cite{AM21,AM21_2} showed that every oriented Riemannian manifold of dimension $\leq 5$ is spin$^{\mathbb{H}}$, and for closed manifolds, this is true for dimensions $\leq 7$. In~\cite{quaternionic_monopoles}, in analogy to Seiberg-Witten theory, the authors construct quaternionic monopoles, which are associated to spin$^{\mathbb{H}}$ structures on $4$-manifolds.  

Espinosa and Herrera~\cite{H1} introduced the concept of spin$^r$ structure for each $r \in \mathbb{N}$, the cases $r=1,2,3$ corresponding to spin, spin$^{\mathbb{C}}$ and spin$^{\mathbb{H}}$ respectively. They call these structures \emph{spinorially twisted spin structures}~\cite{H1,H2,H3}. Albanese and Milivojevi\'c~\cite{AM21} call these structures \emph{generalised spin$^k$ structures}. An oriented Riemannian manifold $M$ is spin$^r$ if and only if it can be embedded into a spin manifold with codimension $r$. Consequently, if $r<s$, every spin$^r$ manifold is spin$^s$. They proved that there is no $r \in \mathbb{N}$ such that every oriented Riemannian manifold is spin$^r$. This result can be interpreted as ensuring that this chain of structures does not \emph{stabilise}. As pointed out by Lawson in~\cite{lawson}, this is remarkable, as every manifold embeds into an orientable manifold with codimension $1$.

For a general spin$^r$ manifold, there is not a preferred spin$^r$ structure to work with. In~\cite{H1}, the authors give a canonical construction of a spin$^n$ structure for every oriented Riemannian $n$-manifold in terms of its holonomy. We show that this structure is not induced by a spin$^s$ structure for any $s < n$ for a wide class of manifolds. 

For $r > 3$, there are presently no known topological obstructions to the existence of a spin$^r$ structure, and the problem of their classification is beyond our reach. In this paper, we give a complete study of invariant spin$^r$ structures on homogeneous spaces. 

Homogeneous spaces are of great interest in this context, because there is a simple way of looking at spin structures provided they are invariant, namely as lifts of the isotropy representation to the spin group~\cite{Bar, AC, DKL}. However, not all homogeneous spaces are spin, and indeed, all the examples provided above are homogeneous. It is therefore natural to ask whether this isotropy-lifting criterion can be extended to invariant spin$^r$ structures. It turns out that this is the case, as we show in Section~\ref{section:invariance}, along with a characterisation theorem (Theorem~\ref{cor:inv_gen_spin}), which extends some of the results in~\cite{CGT}. 

Another feature of homogeneous spaces is that they provide an ideal setting for explicit spinorial computations -- see~\cite{AH24}. For instance, in~\cite{AHL}, the authors explore the relationship between certain invariant spinors on homogeneous spheres and various $G$-structures. 

For each $n \in \mathbb{N}$, we provide simple examples of $n$-dimensional oriented Riemannian homogeneous $G$-spaces $M$ for which the minimal $r\in\mathbb{N}$ such that $M$ admits a $G$-invariant spin$^r$ structure is exactly $n$ (Theorem~\ref{thm:spheres_so}). This establishes a $G$-invariant analogue of the fact that, for each $r\in\mathbb{N}$, there exists an oriented Riemannian manifold that does not admit a spin$^r$ structure~\cite{AM21}. 

Additionally, in Section~\ref{sec:spheres} we study a family of examples. The compact and connected Lie groups that act transitively and effectively on spheres were classified by Montgomery and Samelson in~\cite{class_spheres}. For each of these homogeneous realisations of the spheres, we compute the minimal $r\in \mathbb{N}$ such that the corresponding sphere admits an invariant spin$^r$ structure, refining the picture presented in~\cite{DKL}, where only spin structures were considered. Finally, we show the beautiful relationship between the existence of invariant spin$^r$ structures on homogeneous spheres and lifts of the holonomy representation of an oriented Riemannian $n$-manifold to $\Spin^r(n)$, via Berger's classification of holonomy groups. 

\section{Generalised \texorpdfstring{spin$^r$}{spinʳ} structures}

\subsection{\texorpdfstring{Spin$^r$}{Spinʳ} groups}
Denote by $\SO(n)$ the special orthogonal group, and let $\lambda_n \colon \Spin(n) \to \SO(n)$ be the standard two-sheeted covering. For $r \in \mathbb{N}$, we define the group 
\[
\Spin^r(n) \coloneqq \left(\Spin(n) \times \Spin(r) \right) / \mathbb{Z}_2 \, ,     
\]
where $\mathbb{Z}_2 = \langle(-1,-1)\rangle \subseteq \Spin(n) \times \Spin(r)$. Note that $\Spin^1(n) = \Spin(n)$. Moreover, as $\Spin(2) \cong \U(1)$ and $\Spin(3) \cong \Sp(1)$, it is clear that  $\Spin^2(n) = \Spin^{\mathbb{C}}(n)$ and $\Spin^3(n) = \Spin^{\mathbb{H}}(n)$. There are natural Lie group homomorphisms 
\[
\begin{tabularx}{\linewidth}{@{}XX@{}}
    \flushright
    $\begin{aligned}
        \lambda^{r}_n \colon &\Spin^{r}(n) \to \SO(n) \\
        &[\mu,\nu] \mapsto \lambda_n(\mu)
    \end{aligned}$ \, ,
    &
    \flushleft
    $\begin{aligned}
        \xi^{r}_n \colon &\Spin^{r}(n) \to \SO(r) \\
        &[\mu,\nu] \mapsto \lambda_r(\nu)
    \end{aligned}$ \, . 
\end{tabularx}
\]
For $r<s$, there is a natural inclusion $\iota^{rs}_n \colon \Spin^r(n) \hookrightarrow \Spin^s(n)$, which makes the following diagram commute: 
\[
\begin{tikzcd}[row sep=1cm, column sep=1cm]
    &  &\Spin^{r} (n) \arrow{dr}{\lambda^{r}_n} \arrow[hookrightarrow]{dd}[pos=.2]{\iota^{rs}_n}  & \\
     &\Spin(n) \arrow{rr}[pos=.2]{\lambda_n} \arrow[hookrightarrow]{ur}{\iota^{1r}_n} \arrow[hookrightarrow]{dr}{\iota^{1s}_n} & &\SO(n) \\
     & &\Spin^{s} (n) \arrow{ur}{\lambda^{s}_n} & 
     \end{tikzcd} \, . 
 \]

 We need to first make some topological considerations, which we shall use later. Define the map
 \begin{align*}
     \varphi^{r,n} \colon &\Spin^r(n) \to \SO(n) \times \SO(r) \, . \\
     &[\mu,\nu] \mapsto (\lambda_n(\mu),\lambda_r(\nu))
     \end{align*}
 The following result shows that this map is a two-sheeted covering, and describes it at the level of fundamental groups. We will use this result to apply the lifting criterion. 
 \begin{prop}\label{prop:fg}
 For $n,r \in \mathbb{N}$, let $\varphi^{r,n}_{\sharp}$ be the map induced by $\varphi^{r,n}$ at the level of fundamental groups. Then, 
 \begin{enumerate}
 \item $\varphi^{r,n}$ is a two-sheeted covering map; 
 \item $\varphi^{2,2}_{\sharp}\left(\pi_1\left( \Spin^2(n)\right) \right)  = \langle(1,\pm 1)\rangle \subseteq \mathbb{Z} \times \mathbb{Z} \cong \pi_1 \left( \SO(2) \times \SO(2)\right)$; 
 \item if $n \geq 3$, then $\varphi^{2,n}_{\sharp} \left( \pi_1 \left(\Spin^2(n)\right)\right) = \langle(1,1)\rangle \subseteq \mathbb{Z}_2 \times \mathbb{Z} \cong \pi_1 \left( \SO(n) \times \SO(2)\right)$;
 \item if $r,n \geq 3$, then $\varphi^{r,n}_{\sharp}\left(\pi_1 \left(\Spin^r(n) \right) \right) = \langle (1,1) \rangle \subseteq \mathbb{Z}_2 \times \mathbb{Z}_2 \cong \pi_1 \left( \SO(n) \times \SO(r)\right)$. 
 \end{enumerate}
 \end{prop}
 \begin{proof}
Consider the composition 
 \[
 \Spin(n) \times \Spin(r) \xrightarrow[]{p} \Spin^r(n) \xrightarrow[]{\varphi^{r,n}} \SO(n) \times \SO(r) \, ,     
 \] 
 where $p$ is the obvious projection. 
 
 Point (1) follows from the fact that $\lambda_n,\lambda_r$ and $p$ are two-sheeted covering maps. 
 
 For (2), suppose $r=n=2$. Recall that $\Spin(2) \cong \SO(2) \cong S^1$, and that the map $\varphi^{2,2} \circ p$ is given by $(z,w) \mapsto (z^2,w^2)$. Consider the loop $\alpha_{\pm} \colon [0,1] \to \Spin^2(2)$ defined by $\alpha_{\pm}(t) = p\left( e^{i\pi t} , e^{ \pm i\pi t} \right)$. It is clear that the image of $\alpha_{\pm}$ in the fundamental group of $\SO(2) \times \SO(2)$, which is isomorphic to $\mathbb{Z} \times \mathbb{Z}$, is $(1,\pm 1)$. Moreover, the lift of every loop $\beta$ based at $[1,1]$ in $\Spin^2(2)$ along $p$ starting at $(1,1)$ is of the form $(\beta_1,\beta_2)$, where either both $\beta_i$ are loops based at $1$ or both are paths from $1$ to $-1$ in the corresponding factors of $\Spin(2) \times \Spin(2)$. In the first (resp. second) case, the image of $\beta$ in the fundamental group of $\SO(2) \times \SO(2)$, which is isomorphic to $\mathbb{Z} \times \mathbb{Z}$, is of the form $(2r , 2s)$ (resp. $(2r+1,2s+1)$), for some $r,s \in \mathbb{Z}$. In both cases, the image of $\beta$ is an element of $\langle (1 , \pm 1 ) \rangle$. This proves (2). 
 
 Let us prove (3). Suppose $n \geq 3$. Let $\beta$ be a loop in $\Spin^2(n)$ based at $[1,1]$. Its lift along $p$ starting at $(1,1)$ is a path with endpoint either $(1,1)$ or $(-1,-1)$. Hence, this lift is of the form $(\beta_1,\beta_2)$, where either both $\beta_i$ are loops based at $1$ or both are paths from $1$ to $-1$ in the corresponding factors of $\Spin(n) \times \Spin(2)$. In the first case, the image of $\beta$ in the fundamental group of $\SO(n) \times \SO(2)$, which is isomorphic to $\mathbb{Z}_2 \times \mathbb{Z}$, is of the form $(0,2m)$, for some $m \in \mathbb{Z}$. In the other case, the image of $\beta$ is of the form $(1,2l+1)$, for some $l \in \mathbb{Z}$. This proves (3). 
 
 The proof of (4) is analogous.

 \end{proof}

 \subsection{\texorpdfstring{Spin$^r$}{Spinʳ} structures on manifolds}

As in the classical spin, spin$^{\mathbb{C}}$ and spin$^{\mathbb{H}}$ cases, one can define a family of structures on oriented Riemannian manifolds: 
\begin{deff}
    Let $M$ be an oriented Riemannian $n$-manifold with principal $\SO(n)$-bundle of positively oriented orthonormal frames $FM$. A spin$^r$ structure on $M$ is a lift of the structure group of $FM$ along the Lie group homomorphism $\lambda^r_n$. In other words, it is a pair $(P,\Phi)$ consisting of 
    \begin{itemize}
        \item a principal $\Spin^r(n)$-bundle $P$ over $M$, and
        \item a $\Spin^r(n)$-equivariant bundle homomorphism $\Phi \colon P \to FM$, where $\Spin^r(n)$ acts on $FM$ via $\lambda^r_n$. 
    \end{itemize}
The associated principal $\SO(r)$-bundle to $P$ along $\xi^r_n$ is called the auxiliary bundle of the spin$^r$ structure, and is denoted by $\widehat{P}$. The natural map $\theta \colon P \to FM \widetilde{\times} \widehat{P}$, where $\widetilde{\times}$ denotes fibre product, is a two-sheeted covering. 

If $(P_1,\Phi_1)$ and $(P_2,\Phi_2)$ are spin$^r$ structures on $M$, an equivalence of spin$^r$ structures from $(P_1,\Phi_1)$ to $(P_2,\Phi_2)$ is a $\Spin^r(n)$-equivariant diffeomorphism $f \colon P_1 \to P_2$ such that $\Phi_1 = \Phi_2 \circ f$. 
\end{deff}

\begin{rem}
For $r<s$, a spin$^r$ structure naturally induces a spin$^s$ structure, namely via the associated bundle construction along $\iota^{rs}$. 
\end{rem}

Given an oriented Riemannian manifold $M$, it is therefore natural to ask which is the minimal $r \in \mathbb{N}$ such that $M$ has a spin$^r$ structure. 

\begin{deff}\label{deff:spin_type}
Let $M$ be an oriented Riemannian manifold. The spin type of $M$, denoted by $\Sigma(M)$, is the smallest $r \in \mathbb{N}$ such that $M$ admits a spin$^r$ structure. 
\end{deff}
The spin type measures the failure of a manifold to be spin. 
Being spin$^r$ turns out to be a topological property, as shown in the next result. 

\begin{prop}(\cite[Prop.~3.2]{AM21})\label{prop:characterisation_gen_spin}
For an orientable Riemannian manifold $M$, the following are equivalent:  
\begin{enumerate}
\item $M$ is spin$^r$, 
\item there is an orientable rank-$r$ real vector bundle $E \to M$ such that $TM \oplus E$ is spin, 
\item $M$ immerses in a spin manifold with codimension $r$, 
\item $M$ embeds in a spin manifold with codimension $r$. \qed 
\end{enumerate} 
\end{prop}

This lets us prove that Definition~\ref{deff:spin_type} is a good definition, in the sense that the spin type is defined for all oriented Riemannian manifolds: 

\begin{cor}\label{cor:spin_type_bounded}
An oriented Riemannian $n$-manifold $M$ has spin type $1\leq\Sigma(M)\leq n$.  
\end{cor}
\begin{proof}
    By Whitney's embedding theorem, such a manifold can be smoothly embedded into $\mathbb{R}^{2n}$, which is spin. Now apply Proposition~\ref{prop:characterisation_gen_spin}. Equivalently, one can take $E = TM$ in point (2) of Proposition~\ref{prop:characterisation_gen_spin}, and note that $w_1(TM \oplus TM)=0$ and $w_2(TM \oplus TM)=0$, as $TM$ is oriented.    
\end{proof}

\begin{rem}\label{remark:cohen}
    Note that, if $n>1$ and $M$ is a closed oriented Riemannian $n$-manifold, one can improve the bound in Corollary~\ref{cor:spin_type_bounded} to $1 \leq \Sigma(M) \leq n-\alpha(n)$, where $\alpha(n)$ is the number of ones in the binary expression of $n$, by Cohen's theorem~\cite{cohen}. 
\end{rem}

There is a way of constructing an explicit spin$^n$ structure on any connected and oriented Riemannian $n$-manifold $M$. Fix $x \in M$, and let $G$ be the holonomy group of $M$ at $x$. Let $h \colon G \to \SO(n)$ be the holonomy representation at $x$. The positively oriented orthonormal frame bundle of $M$ admits a lift of the structure group from $\SO(n)$ to $G$ via $h$. Hence, if $h$ lifts to $\Spin^r(n)$, we obtain a spin$^r$ structure on $M$ via the associated bundle construction. This is always the case for $r=n$: 

\begin{prop}(see~\cite[Prop.~3.3]{H1})\label{prop:hol}
    Let $M$ be a connected oriented Riemannian $n$-manifold, let $x\in M$, and let $h \colon G \to \SO(n)$ be the holonomy representation at $x$. Then, $h$ lifts to $\Spin^n(n)$, i.e., there exists a Lie group homomorphism $\widetilde{h} \colon G \to \Spin^n (n)$ such that the diagram
    \[ 
    \begin{tikzcd}[row sep=1cm, column sep=1cm]
        &  & \Spin^n (n) \arrow{d}{\lambda_n^n}\\
        & G \arrow{r}{h} \arrow{ur}{\widetilde{h}} &\SO(n)
        \end{tikzcd}
    \]
    commutes.  
\end{prop}
\begin{proof}
 Recall that $\lambda_{n}^{n} = \pr_1 \circ \varphi^{n,n}$. The map $h \times h \colon G \to \SO(n) \times \SO(n)$ always lifts to $\Spin^n(n)$ along $\varphi^{n,n}$, by Proposition~\ref{prop:fg}. 
\end{proof}

This construction gives us a naturally preferred spin$^n$ structure on any connected oriented Riemannian $n$-manifold. Note that its auxiliary vector bundle is the tangent bundle $TM$ itself (compare with the proof of Corollary~\ref{cor:spin_type_bounded}). It is natural to ask whether this spin$^n$ structure on $M^n$ is \emph{proper}, in the sense that it is not induced by a spin$^r$ structure on $M$, for some $r<n$. To answer this question, we first generalise a result by Moroianu~\cite[Lem.~2.1]{moroianu}, who focused on the case $r=2$, $s=1$:  
\begin{lemma}\label{lemma:moroianu_gen}
    Let $M$ be an oriented Riemannian $n$-manifold admitting a spin$^r$ structure for some $r\geq 1$, and let $1 \leq s \leq r$. Then, the spin$^r$ structure is induced by a spin$^s$ structure along $\iota^{sr}_n$ if and only if the auxiliary $\SO(r)$-bundle of the spin$^r$ structure admits a reduction of the structure group to $\SO(s)$. 
\end{lemma}
\begin{proof}
The proof generalises the one in~\cite{moroianu}. One implication follows easily from the commutativity of the diagram 
\[ 
    \begin{tikzcd}[row sep=1cm, column sep=1cm]
        & \Spin^s(n) \arrow{d}{\xi_{n}^{s}} \arrow{r}{\iota_{n}^{sr}}  & \Spin^r (n) \arrow{d}{\xi_{n}^{r}}\\
        & \SO(s) \arrow{r}{j^{sr}}  &\SO(r)
        \end{tikzcd} \, ,
\]
where $j^{sr}$ is the inclusion of $\SO(s)$ into $\SO(r)$ as the top left-hand block. For the other implication, let $\Sigma$ be a reduced $\SO(s)$-subbundle of the auxiliary bundle $\widehat{P}$ of the spin$^r$-structure $(P,\Phi)$. Then, the preimage of $FM \widetilde{\times} \Sigma$ under the map $\theta \colon P \to FM \widetilde{\times} \widehat{P}$ is a spin$^s$ structure from which the original spin$^r$ structure is induced. 
\end{proof}
 
A principal $\SO(r)$-bundle admits a reduction of the structure group to $\SO(s)$ if and only if the associated rank-$r$ vector bundle admits $r-s$ global orthonormal sections. In the construction of Proposition~\ref{prop:hol}, the auxiliary rank-$n$ vector bundle is the tangent bundle $TM$. Take, for instance, an even-dimensional sphere $S^{2n}$. By the hairy ball theorem, it does not admit a nowhere-vanishing vector field, so in particular the spin$^{2n}$ structure constructed in Proposition~\ref{prop:hol} is not induced by any spin$^r$ structure, for any $r < 2n$. This argument generalises to a wide class of manifolds, namely those closed oriented Riemannian manifolds with non-zero Euler characteristic class. Examples of these include $S^{2n}$, $\mathbb{CP}^n$, $\mathbb{HP}^n$ and $\mathbb{OP}^2$. For more homogeneous examples, see~\cite{Wang49}. Note also that every odd-dimensional manifold admits a nowhere-vanishing vector field, and hence the spin$^n$ structure given by the holonomy construction of Proposition~\ref{prop:hol} is induced by a spin$^{n-1}$ structure. 

\begin{cor}
Let $M$ be an oriented Riemannian $n$-manifold. If $r >n$, every spin$^r$ structure on $M$ is induced by a spin$^n$ structure on $M$ along the inclusion $\iota_{n}^{n,r} \colon \Spin^n(n) \hookrightarrow \Spin^r(n)$. 
\end{cor}
\begin{proof}
Recall that, if $r > n$, every rank-$r$ vector bundle over an $n$-manifold has a nowhere-vanishing section~\cite{hirsch}. Apply this inductively to the auxiliary vector bundle of a spin$^r$ structure on $M$. 
\end{proof}

Using Lemma~\ref{lemma:moroianu_gen}, one can refine Corollary~\ref{cor:spin_type_bounded}: 

\begin{cor}\label{cor:improvement}
Let $n > 1$. An oriented Riemannian $n$-manifold $M$ has spin type $1 \leq \Sigma(M) \leq n-1$. 
\end{cor}
\begin{proof}
If $M$ is compact, this follows from Cohen's theorem, as in Remark~\ref{remark:cohen}. If $M$ is non-compact, recall that every rank-$n$ vector bundle over $M$ has a nowhere-vanishing section~\cite{hirsch}. Hence, by Lemma~\ref{lemma:moroianu_gen}, every spin$^n$ structure on $M$ comes from a spin$^{n-1}$ structure on $M$.  
\end{proof}

\section{Invariance} \label{section:invariance}

 We now turn to the case where a Lie group $G$ acts smoothly on a manifold $M$, and study the notion of $G$-invariance of these structures. Later, we will focus on the case where $G$ acts transitively on $M$, making it into a homogeneous space. 
 
The main result of this section is Theorem~\ref{cor:inv_gen_spin}, which gives a classification of $G$-invariant spin$^r$ structures on homogeneous $G$-spaces up to $G$-equivariant equivalence. Its proof relies on Lemmata~\ref{lemma:lift} and~\ref{lemma:lift_gen}. First, we need some preliminaries on equivariant bundles. 

 \subsection{Equivariant bundles}

\begin{deff}
    Let $M$ be a smooth manifold equipped with a left action of a Lie group $G$. A \emph{$G$-equivariant structure} on a principal $K$-bundle $\pi \colon P \to M$ is a left $G$-action on $P$ by bundle homomorphisms such that $\pi$ is $G$-equivariant. A principal bundle with a choice of $G$-equivariant structure is a \emph{$G$-equivariant bundle}. 
\end{deff}

\begin{ex}\label{ex:trivial}
    Let $G/H$ be a homogeneous space, and $K$ a Lie group. 
    \begin{enumerate}
        \item If $\varphi \colon H \to K$ is a Lie group homomorphism, the principal $K$-bundle 
        \[ 
        G \times_{\varphi} K \to G/H \, , \quad [g,k] \mapsto gH
        \]
        has a natural $G$-equivariant structure given by 
        \[ 
        g' \cdot [g,k] := [g'g,k] \, .  
        \]
        In part (1) of Lemma \ref{lemma:lift}, we will see that every $G$-equivariant bundle over a homogeneous $G$-space is of this form. 
        \item $G$-equivariant structures on the trivial bundle $G/H \times K \to G/H$ are of the form 
        \[ 
        g' \cdot (gH , k) := (g'gH , \alpha(g , g') k) \, , 
        \]
        where $\alpha \colon G \times G \to K$ is a smooth map satisfying $\alpha(\cdot,e_G) = e_K$ and
        \[ 
        \forall g , g' , g'' \in G \colon \quad \alpha(g , g' g'') = \alpha(g'' g , g') \alpha(g , g'') \, . 
        \]
        If $\alpha$ is independent of the first argument, we call the corresponding $G$-equivariant structure \emph{special}. 
    \end{enumerate}
\end{ex}

\begin{deff}
    Let $P$ be a $G$-equivariant principal bundle over the homogeneous space $G/H$. We say that $P$ is \emph{strongly $G$-trivial} if it is $G$-equivariantly isomorphic to the trivial bundle equipped with a \emph{special} $G$-equivariant structure -- see part (2) of Example~\ref{ex:trivial}. 
\end{deff}

\begin{lemma}\label{lemma:lift}
     Let $G/H$ be a homogeneous space, $K$ a Lie group, and $\widehat{\pi} \colon P \to G/H$ a $G$-equivariant principal $K$-bundle. Then, there exists a Lie group homomorphism $\varphi \colon H \to K$ such that:  
     \begin{enumerate} 
        \item $P$ is $G$-equivariantly bundle-isomorphic to 
        \[ 
        \pi \colon G \times_{\varphi} K \to G/H \, , \quad [g,k] \mapsto gH 
        \, .
        \]
        Moreover, $\varphi$ is unique up to conjugation by elements of $K$. 
        \item $P$ is trivial if and only if $\varphi$ extends to a smooth map $\widetilde{\varphi} \colon G \to K$ satisfying \begin{equation}\label{eq:property_trivial_bundle}
        \forall g \in G , h \in H \colon \quad \widetilde{\varphi}(gh) = \widetilde{\varphi}(g) \widetilde{\varphi}(h) \, . 
        \end{equation}
        \item $P$ is strongly $G$-trivial if and only if $\varphi$ extends to a Lie group homomorphism $\widetilde{\varphi} \colon G \to K$.  
     \end{enumerate}
 \end{lemma}
 \begin{proof} 
    A proof of part (1) can be found in~\cite[Lem.~1.4.5]{CSbook}. As we use ideas from this proof later, we include it here. Note that the Lie group $G \times K$ acts on $P$ by 
     \begin{equation*}
         (g,k) \cdot p \coloneqq g \cdot p \cdot k^{-1} \, , 
     \end{equation*}
     for every $g \in G$, $k \in K$ and $p\in P$. Note that, as the $G$-action on $P$ is by $K$-bundle homomorphisms, this is a well-defined action of $G \times K$ on $P$. Moreover, this action is transitive, as the base space is a homogeneous $G$-space and $K$ acts transitively on each fibre. Fix $p_0 \in P$ with projection $o \coloneqq e_G H \in G/H$. Let $H_{p_0} \subseteq G \times K$ be the stabiliser of $p_0$. The first projection $H_{p_0} \to G$ lands in $H$, because the right $K$-action does not change the fibre. And the resulting homomorphism $H_{p_0} \to H$ is in fact an isomorphism, since the $K$-action is simply transitive on the fibres. Denote by $\widehat{\varphi}$ the inverse isomorphism, and define $\varphi \coloneqq \pr_2 \circ \widehat{\varphi} \colon H \to K$. Now, it is straightforward to check that the map 
     \begin{align*}
         \begin{split}
      G \times_{\varphi} K &\to P\\
         [g,k] &\mapsto g p_0 k 
         \end{split}   
     \end{align*}
     is a $G$-equivariant $K$-bundle isomorphism. 

     For the last assertion, suppose $\varphi_1,\varphi_2 \colon H \to K$ are Lie group homomorphisms, and suppose that there exists a $G$-equivariant $K$-bundle isomorphism $f \colon G \times_{\varphi_1} K \to G \times_{\varphi_2} K$. Then, for each $g \in G$ and $k \in K$,  
     \begin{equation*}
         f([g,k]) = f(g \cdot [e_G,e_K] \cdot k) = g \cdot f([e_G,e_K]) \cdot k \, , 
     \end{equation*}
     and so $f$ is determined by $f([e_G,e_K])$. As $f$ is a bundle homomorphism, $f([e_G,e_K])$ lies over $o$, and hence there exists $t \in K$ such that $f([e_G,e_K]) = [e_G,t]$. Now, for every $g \in G$, $k \in K$ and $h \in H$, we have that 
     \begin{align*}
         [g,tk] &= g \cdot [e_G,t] \cdot k = g \cdot f([e_G,e_K]) \cdot k = f(g \cdot [e_G,e_K] \cdot k) \\
         &= f([g,k]) = f([gh,\varphi_1(h)^{-1} k]) = [gh,t\varphi_1(h)^{-1} k] = [g,\varphi_2(h) t \varphi_1(h)^{-1} k ] \, . 
     \end{align*}
     Therefore, for each $h \in H$, $\varphi_2(h) = t \varphi_1(h) t^{-1}$. 
    
    For part (2), suppose first that $P$ is a trivial bundle, so that there is a bundle isomorphism $F \colon G \times_{\varphi} K \to G/H \times K$. Such a map has the form $[g,k] \mapsto (gH , \widetilde{\varphi}(g) k)$, for some smooth map $\widetilde{\varphi} \colon G \to K$, which we can take so that $\widetilde{\varphi}(e_G) = e_K$. Note that, for $g \in G$, $h \in H$, and $k \in K$,  
    \begin{align*}
     (gH , \widetilde{\varphi}(g) k) = F([g,k]) = F([gh, \varphi(h)^{-1} k]) = (ghH , \widetilde{\varphi}(gh) \varphi(h)^{-1} k) = (gH , \widetilde{\varphi}(gh) \varphi(h)^{-1} k) \, . 
    \end{align*}
    Hence, $\widetilde{\varphi}(gh) = \widetilde{\varphi}(g) \varphi(h)$. In particular, as $\widetilde{\varphi}(e_G) = e_K$, $\widetilde{\varphi}$ is an extension of $\varphi$. 
    
    Conversely, suppose that $\varphi$ admits an extension to $G$ satisfying~\eqref{eq:property_trivial_bundle}. Then, we can trivialise the bundle $G \times_{\varphi} K$ via the map $F$ defined above. 

    Let us now prove part (3). If such a homomorphism $\widetilde{\varphi}$ exists, we can equip the trivial bundle $G/H \times K$ with the special $G$-action 
    \[ 
    g' \cdot (gH , k) := (g'gH , \widetilde{\varphi}(g') k) \, , 
    \]
    and the proof of (1) shows that this bundle is $G$-equivariantly isomorphic to $P$. 
    
    Conversely, suppose that $P$ is strongly $G$-trivial. Then, $P$ is $G$-equivariantly isomorphic to the trivial bundle $G/H \times K$ equipped with a special $G$-equivariant structure, i.e., one given by 
    \[ 
    g' \cdot (gH , k) := (g'g H , \alpha(g') k) \, , 
    \]
    for a Lie group homomorphism $\alpha \colon G \to K$. Running the proof of (1) for the trivial bundle with this $G$-action, we see that this bundle is also $G$-equivariantly isomorphic to $G \times_{\alpha \restr{H}} K$. Hence, by the final assertion of part (1), there exists $t \in K$ such that $\alpha \restr{H} = t \varphi t^{-1}$. Taking $\widetilde{\varphi} = t^{-1} \alpha t$ finishes the proof.  
 \end{proof}
 
 \begin{ex}
 If $G/H$ is an $n$-dimensional oriented Riemannian homogeneous space, the smooth action of $G$ on $G/H$ induces a smooth action of $G$ on the positively oriented orthonormal frame bundle of $G/H$ by $\SO(n)$-bundle homomorphisms, covering the $G$-action on $G/H$. The Lie group homomorphism $\varphi \colon H \to \SO(n)$ given by Lemma~\ref{lemma:lift} is precisely the isotropy representation. 
 \end{ex}
 
 We can generalise Lemma~\ref{lemma:lift}: 
 
 \begin{lemma}\label{lemma:lift_gen}
      Let $G/H$ be a homogeneous space. Let $K,L$ be Lie groups. Let $\lambda \colon K \to L$ and $\eta \colon H \to L$ be Lie group homomorphisms, with $\lambda$ surjective. Let $P$ be a $G$-equivariant lift of the structure group of $G \times_{\eta} L$ to $K$ along $\lambda$. Then, there exists a Lie group homomorphism $\widetilde{\eta} \colon H \to K$ such that $P$ is $G$-equivariantly bundle-isomorphic to $G \times_{\widetilde{\eta}} K$ covering the identity on $G \times_{\eta} L$, and $\eta = \lambda \circ \widetilde{\eta}$. Moreover, $\widetilde{\eta}$ is unique up to conjugation by elements of $\ker(\lambda)$. 
     \end{lemma}
     \begin{proof}
         Let $f \colon P \to G \times_{\eta} L$ be a $G$-equivariant lift of the structure group along $\lambda$. As $\lambda$ is surjective, there exists $p_0 \in P$ such that $f(p_0) = [e_G,e_L]$. Applying the same argument of the proof of part (1) of Lemma~\ref{lemma:lift} with this choice of $p_0$, one obtains a Lie group homomorphism $\widetilde{\eta}\colon H \to K$ such that $P$ is $G$-equivariantly bundle-isomorphic to $G \times_{\widetilde{\eta}} K$, and we can take this isomorphism to send $p_0$ to $[e_G,e_K]$. Hence, we have a $G$-equivariant lift of the structure group 
         \begin{align*}
             \begin{split}
                 G \times_{\widetilde{\eta}} K & \overset{\cong}{\longrightarrow} P \to G \times_{\eta} L\\
             [g,k] &\mapsto g p_0 k \mapsto [g,\lambda(k)] \, . 
             \end{split}   
         \end{align*}
         One readily checks that this map is well-defined if and only if $\lambda \circ \widetilde{\eta} = \eta$. 

         Now, suppose that $\widetilde{\eta}_1 , \widetilde{\eta}_2$ are such that 
         \[ 
            \lambda \circ \widetilde{\eta}_1 = \lambda \circ \widetilde{\eta}_2 = \eta \, . 
         \]
         Suppose further that there exists a $G$-equivariant $K$-bundle isomorphism $f$ such that the diagram
         \[
             \begin{tikzcd}[row sep=1cm, column sep=1.5cm]
                 & G \times_{\widetilde{\eta}_1} K \arrow{rr}{f} \arrow{dr}[swap]{} &  & G \times_{\widetilde{\eta}_2} K \arrow{dl}[swap]{}\\
                 & & G \times_{\eta} L & 
                 \end{tikzcd}    
         \]
         commutes. Then, there exists $t \in \ker(\lambda)$ such that 
         \[ 
                f \left( \left[ e_G , e_K \right] \right) = \left[ e_G , t \right] \, . 
         \]
         Hence, for every $h \in H$, 
         \[ 
            \left[ e_G , t \right] = f \left( \left[ e_G , e_K \right] \right) = f \left( \left[ h , \widetilde{\eta}_1 (h)^{-1} \right] \right) = \left[ h , t \widetilde{\eta}_1 (h)^{-1} \right] = \left[ e_G , \widetilde{\eta}_2 (h) t \widetilde{\eta}_1 (h)^{-1} \right] \, . 
         \] 
         This shows that 
         \[ 
                 \widetilde{\eta}_2 = t \widetilde{\eta}_1 t^{-1} \, . 
         \]
        Conversely, if $\widetilde{\eta}_1$ and $\widetilde{\eta}_2$ are conjugate by an element $t \in \ker(\lambda)$, we can define $f$ by 
        \[ 
            f \left( \left[ g , k \right] \right) = \left[ g , t k \right] \, . 
        \]
     \end{proof}
 
 \begin{rem}
 We recover Lemma~\ref{lemma:lift} by taking $L$ to be the trivial one-element Lie group. 
 \end{rem}
 
 \begin{ex}
     Taking $\lambda = \lambda_n \colon \Spin(n) \to \SO(n)$ and $\eta = \sigma$ the isotropy representation in Lemma~\ref{lemma:lift_gen}, we recover~\cite[Prop.~1.3]{DKL}. 
\end{ex}

 \subsection{Invariant \texorpdfstring{spin$^r$}{spinʳ} structures}

Let $G$ be a Lie group acting by orientation-preserving isometries on an oriented Riemannian $n$-manifold $M$ with positively oriented orthonormal frame bundle $FM$.

\begin{deff}
A \emph{$G$-invariant spin$^r$ structure} on $M$ is a spin$^r$ structure $(P,\Phi)$ where $P$ and $\Phi$ are $G$-equivariant.   

Two $G$-invariant spin$^r$ structures $\left( P_1 , \Phi_1 \right), \left( P_2 , \Phi_2 \right)$ on $M$ are said to be \emph{$G$-equivariantly equivalent} if there exists a $G$-equivariant $\Spin^r(n)$-bundle isomorphism $f \colon P_1 \to P_2$ making the following diagram commute: 
\[
    \begin{tikzcd}[row sep=1cm, column sep=1cm]
        & P_1 \arrow{rr}{f} \arrow{dr}[swap]{\Phi_1} &  & P_2 \arrow{dl}{\Phi_2}\\
        & & FM & 
    \end{tikzcd}    \, . 
\]
 \end{deff}

Finally, we can state our main result: 
 
     \begin{thmm}\label{cor:inv_gen_spin}
         Let $G/H$ be an $n$-dimensional oriented Riemannian homogeneous space with $H$ connected and isotropy representation $\sigma \colon H \to \SO(n)$. Then, there is a bijective correspondence between 
         \begin{itemize}
             \item $G$-invariant spin$^r$ structures on $G/H$ modulo $G$-equivariant equivalence of spin$^r$ structures, and
             \item Lie group homomorphisms $\varphi \colon H \to \SO(r)$ such that $\sigma \times \varphi \colon H \to \SO(n) \times \SO(r)$ lifts to $\Spin^r(n)$ along $\lambda_n^r$ modulo conjugation by an element of $\SO(r)$. 
         \end{itemize}
         \end{thmm}
         \begin{proof}
             A $G$-invariant spin$^r$ structure is, by definition, a $G$-equivariant lift of the structure group of the positively oriented orthonormal frame bundle $G \times_{\sigma} \SO(n)$. Hence, by Lemma~\ref{lemma:lift_gen}, $G$-invariant spin$^r$ structures on $G/H$ modulo $G$-equivariant equivalence of spin$^r$ structures are in bijection with Lie group homomorphisms $\widetilde{\sigma} \colon H \to \Spin^r(n)$ such that $\lambda_{n}^r \circ \widetilde{\sigma} = \sigma$ modulo conjugation by elements of $\ker(\lambda_n^r)$. But, as $H$ is connected and $\lambda_n^r$ is a covering map, these are in bijection with Lie group homomorphisms $\varphi \colon H \to \SO(r)$ such that $\sigma \times \varphi$ lifts to $\Spin^r(n)$ modulo conjugation by elements of $\SO(r)$. 
         \end{proof}
 
         \begin{rem}
             For the case $r=1$, as $\SO(1)$ is trivial, we get two cases: either $\sigma$ lifts to $\Spin(n)$ or it does not. So we have either a unique such structure or none. This recovers the well-known result that, if a $G$-invariant spin structure exists, it is unique (see e.g.~\cite[Cor.~1.4]{DKL}). 
         \end{rem}

 In the light of Theorem~\ref{cor:inv_gen_spin}, it is clear that, if $r < s$, a $G$-invariant spin$^r$ structure naturally induces a $G$-invariant spin$^s$ structure. So it makes sense to introduce a $G$-invariant analogue of the spin type of Definition~\ref{deff:spin_type}: 

\begin{deff}
Let $G$ be a Lie group acting smoothly by orientation-preserving isometries on an oriented Riemannian manifold $M$. The $G$-invariant spin type of $M$, denoted by $\Sigma(M,G)$, is the least $r \in \mathbb{N}$ such that $M$ admits a $G$-invariant spin$^r$ structure. 
\end{deff}
For example, we know that the unique spin structure of $S^n$ is not $G$-invariant for $G=\SO(n+1),\U((n+1)/2)$~\cite{DKL}, i.e., 
\[ \Sigma(S^n,\SO(n+1)), \, \Sigma(S^n,\U((n+1)/2)) > 1\, . \] 
But there might exist a $G$-invariant spin$^{\mathbb{C}}$ structure. Let us explore this in the next section. 

\section{Spheres} \label{sec:spheres}

The aim of this section is to find the $G$-invariant spin type $\Sigma \left( S^n , G \right)$ of $S^n$ for each $n \in \mathbb{N}$ and each compact and connected Lie group $G$ acting smoothly, transitively and effectively by orientation-preserving isometries on $S^n$. These groups were classified in~\cite{class_spheres}. In~\cite{DKL}, the authors determine when $\Sigma \left( S^n , G \right) = 1$. We will compute the rest (see Table~\ref{table:invariant_spin_type}). 

\begin{thmm}\label{thm:spheres_so}
The $\SO(n+1)$-invariant spin type of $S^n$ is $n$, for $n \neq 4$, and $3$, for $n$=4 (compare with~\cite[Prop.~3.9]{AM21}). In other words, 
\begin{equation*}
    \Sigma(S^n,\SO(n+1)) = 
    \begin{dcases} 
    n, & n \neq 4\\
    3, & n=4 \, .
 \end{dcases} 
\end{equation*}
Moreover, the $\SO(n+1)$-invariant spin$^n$ structures on $S^n$ up to $\SO(n+1)$-equivariant equivalence are classified as follows: 
\begin{enumerate}
\item $n=1$: there is a unique $\SO(2)$-invariant spin structure on $S^1$. 
\item $n=2$: the $\SO(3)$-invariant spin$^{\mathbb{C}}$ structures on $S^2$ are the ones corresponding to $\varphi_s \colon S^1 \cong \SO(2) \to \SO(2) \cong S^1$ given by $z \mapsto z^s$, with $s$ odd. 
\item $n \geq 3$ odd: there exists a unique $\SO(n+1)$-invariant spin$^n$ structure on $S^n$, and it is the one corresponding to the identity map $\SO(n) \to \SO(n)$.  
\item $n=4$: There are two $\SO(5)$-invariant spin$^{\mathbb{H}}$ structures on $S^4$, corresponding to the Lie group homomorphisms 
\[ 
    \SO(4) \to \SO(4) / \mathbb{Z}_2 \cong \SO(3) \times \SO(3) \overset{\pr_i}{\to} \SO(3)
\]
for $i = 1,2$. 
\item $n \geq 6$ even: there exist two different $\SO(n+1)$-invariant spin$^n$ structures on $S^n$, namely the ones corresponding to the identity map $\SO(n) \to \SO(n)$ and to conjugation by a fixed element of $\Ort(n) \setminus \SO(n)$. 
\end{enumerate}
\end{thmm}
\begin{proof}
    We can build an $\SO(n+1)$-invariant spin$^n$ structure on $S^n$ as follows. Consider the diagram 
    \[ 
    \begin{tikzcd}[row sep=1cm, column sep=1cm]
        &  & & \Spin^n (n) \arrow{d}{\varphi^{n,n}}\\
        & \SO(n) \arrow{r}{\sigma=id} &\SO(n) \arrow{r}{\Delta} &\SO(n) \times \SO(n)
        \end{tikzcd} \, , 
    \]
    where $\Delta$ is the diagonal map. By Proposition~\ref{prop:fg}, $\Delta \circ \sigma$ lifts to $\Spin^n(n)$, yielding an $\SO(n+1)$-invariant spin$^n$ structure on $S^n$.

Suppose $n \neq 1,2,4$. Let $2 \leq r < n$, and suppose $S^n$ has an $\SO(n+1)$-invariant spin$^r$ structure. Then, there exists a lift $\widetilde{\sigma} \colon \SO(n) \to \Spin^r(n)$, making the following diagram commute: 
\[
\begin{tikzcd}[row sep=1cm, column sep=1cm]
    & & \Spin^r (n) \arrow{dr}{\varphi^{r,n}} &\\
    & \SO(n) \arrow{dr}[swap]{\sigma=id} \arrow{ur}{\widetilde{\sigma}} & & \SO(n) \times \SO(r) \arrow{dl}[near start]{\pr_1} \\
    & & \SO(n) &
\end{tikzcd} \, .
\]

Now consider the Lie group homomorphism $\pr_2 \circ \varphi^{r,n} \circ \widetilde{\sigma} \colon \SO(n) \to \SO(r)$. This induces a Lie algebra homomorphism $\mathfrak{so}(n) \to \mathfrak{so}(r)$. The kernel of this map is an ideal of $\mathfrak{so}(n)$, which is a simple Lie algebra (as $ n \geq 3$ and $n \neq 4$), and hence it is either the whole of $\mathfrak{so}(n)$ or $0$. As $r < n$, it cannot be $0$ for dimensional reasons. Hence, it has to be the whole of $\mathfrak{so}(n)$. As $\SO(n)$ is connected, the Lie exponential map generates the whole group, and hence the Lie group homomorphism has to be trivial. Now, let $\alpha_k$ be a generator of the fundamental group of $\SO(k)$. Then, $\alpha_n \in \pi_1(\SO(n))$ goes, on the one hand, to $(\alpha_n , 0) \in \pi_1(\SO(n) \times \SO(r))$ and, on the other hand, to $(\alpha_n, s\alpha_r)$, for some odd $s$, by Proposition~\ref{prop:fg}, which is a contradiction.  

Now let us consider the case $n=4$. In this case, the Lie algebra $\mathfrak{so}(4) \cong \mathfrak{so}(3) \oplus \mathfrak{so}(3)$, so it is not simple, and the previous argument does not work. First, let us show that $S^4$ does not admit an $\SO(5)$-invariant spin$^2$ structure. Indeed, if such a structure existed, as before, we would have a Lie group homomorphism $\SO(4) \to \SO(2)$, and hence a Lie algebra homomorphism $\mathfrak{so}(3) \oplus \mathfrak{so}(3) \to \mathfrak{so}(2)$. The kernel of this map is an ideal of $\mathfrak{so}(3) \oplus \mathfrak{so}(3)$, and so it is either $0$, $\mathfrak{so}(3)$ or the whole of $\mathfrak{so}(3) \oplus \mathfrak{so}(3)$. It cannot be $0$ for dimensional reasons. If the kernel is the whole Lie algebra, then the Lie group homomorphism has to be trivial, as before, and hence concluding that the isotropy cannot lift to $\Spin^2(4)$, by Proposition~\ref{prop:fg}. So, the only possibility is that the kernel is $\mathfrak{so}(3)$, which again is impossible for dimensional reasons. Hence, the $\SO(5)$-invariant spin type of $S^4$ is either $3$ or $4$. We shall show now that it is $3$.

We know that $\SO(4)/\mathbb{Z}_2 \cong \SO(3) \times \SO(3)$. Hence, we have a chain of Lie group homomorphisms: 
\[ 
    \SO(4) \xrightarrow{\pi} \SO(4)/\mathbb{Z}_2 \xrightarrow{\cong} \SO(3) \times \SO(3) \xrightarrow{\pr_i} \SO(3) \, . 
\] 
At the level of fundamental groups, a generator of $\pi_1 \left( \SO(4) \right)$ goes to a generator of $\pi_1\left(\SO(3)\right)$, for a suitable choice of $i \in \{1,2\}$. Hence, the isotropy representation lifts to $\Spin^3(4)$. 

The rest of the theorem follows from standard facts of representation theory together with Proposition~\ref{prop:fg} and Theorem~\ref{cor:inv_gen_spin}. 
\end{proof}

This is an instance of a more general result: 

\begin{thmm}\label{thm:number_bounded}
Let $n \geq 3$ and let $G/H$ be an $n$-dimensional oriented Riemannian homogeneous space. Then, there is either a canonical $G$-invariant spin structure on $G/H$ or a canonical $G$-invariant spin$^n$ structure on $G/H$. In particular, its $G$-invariant spin type satisfies $1 \leq \Sigma(G/H,G) \leq n$. 
\end{thmm}
\begin{proof}
Let $\sigma \colon H \to \SO(n)$ be its isotropy representation. The induced homomorphism at the level of fundamental groups is $\sigma_{\sharp} \colon \pi_1(H) \to \pi_1 \left( \SO(n) \right) \cong \mathbb{Z}_2$. If $\sigma_{\sharp}$ is trivial, then the isotropy representation lifts to $\Spin(n)$, and hence $G/H$ has a (unique) $G$-invariant spin structure. If $\sigma_{\sharp}$ is not trivial, then consider $\sigma \times \sigma \colon H \to \SO(n) \times \SO(n)$. Then, $\left( \sigma \times \sigma \right)_{\sharp} \left( \pi_1(H)\right) \subseteq \langle(1,1)\rangle \subseteq \mathbb{Z}_2 \times \mathbb{Z}_2 \cong \pi_1 \left( \SO(n) \times \SO(n) \right)$, and $\left(\varphi^{n,n}\right)_{\sharp} \left( \pi_1 \left(\Spin^n(n)\right)\right) = \langle (1,1) \rangle$, by Proposition~\ref{prop:fg}. Hence, the isotropy representation lifts to $\Spin^n(n)$, yielding a $G$-invariant spin$^n$ structure on $G/H$. 
\end{proof}

Note that Theorem~\ref{thm:spheres_so} shows that the bound in Theorem~\ref{thm:number_bounded} is sharp (compare with Corollary~\ref{cor:improvement}). Moreover, recall that, if the isotropy $H$ is connected and $G/H$ has a $G$-invariant spin structure, then such a structure is unique~\cite{DKL}. This has the advantage of singling out a \emph{preferred} spin structure to work with. We observed that, in general, we cannot expect to have a unique $G$-invariant spin$^r$ structure. Theorem~\ref{thm:number_bounded} gives a $G$-invariant spin$^n$ structure which is built in a very natural way using only the isotropy representation. So, in general, an $n$-dimensional oriented Riemannian homogeneous $G$-space comes equipped with a \emph{natural} $G$-invariant spin$^n$ structure.  

Let us now continue with the study of homogeneous spheres. We know that the unique spin structure of $S^{2n+1}$ is not $\U(n+1)$-invariant~\cite{DKL}. So, the $\U(n+1)$-invariant spin type of $S^{2n+1}$ is not $1$. However, we have the following: 

\begin{thmm}\label{thm:spheres_u}
    The $\U(n+1)$-invariant spin type of $S^{2n+1}$ is $2$: 
    \begin{equation*}
        \Sigma(S^{2n+1},\U(n+1)) =2 \, .
    \end{equation*}
    \end{thmm}
    Moreover, the $\U(n+1)$-invariant spin$^{\mathbb{C}}$ structures on $S^{2n+1}$ are the ones corresponding to 
    \begin{align*}
        \varphi_s \colon \U(n) \to \U(1) \\
        A \mapsto \det(A)^s
    \end{align*} 
    with $s$ odd. 
\begin{proof}
The isotropy representation $\sigma \colon \U(n) \to \SO(2n+1)$ is given by the natural inclusions $\U(n) \subseteq \SO(2n) \subseteq \SO(2n+1)$. Consider the map $ \sigma \times \det \colon \U(n) \to \SO(2n+1) \times \SO(2)$, given by $\sigma(A) = (A,\det(A))$. The image of a generator of the fundamental group of $\U(n)$ is an element of the form $(\alpha_{2n+1},\alpha_2)$, where $\alpha_{s}$ is a generator of the fundamental group of $\SO(s)$. Hence, by Proposition~\ref{prop:fg}, this map lifts to $\Spin^{2}(2n+1)$, yielding a $\U(n+1)$-invariant spin$^2$ structure on $S^{2n+1}$.
 
For the final assertion of the theorem, note that the only Lie group homomorphisms $\U(n) \to \SO(2) \cong \U(1)$ are $\varphi_s$, for $s \in \mathbb{Z}$. Moreover,  y Proposition~\ref{prop:fg}, $\sigma \times \varphi_s$ lifts to $\Spin^{\mathbb{C}}(2n+1)$ if and only if $s$ is odd. And, as $\U(1)$ is abelian, the spin$^{\mathbb{C}}$ structures defined by these are pairwise non-$\U(n+1)$-equivariantly equivalent. 
\end{proof}

We know that the $\Sp(n+1)\cdot \U(1)$-invariant spin type of $S^{4n+3}$ is $1$ for $n$ odd, and that it is not $1$ for $n$ even~\cite{DKL}. 

\begin{thmm}\label{thm:spheres_spu}
The $\Sp(2n+1) \cdot \U(1)$-invariant spin type of $S^{8n+3}$ is $2$. Hence, 
\begin{equation*}
    \Sigma(S^{4n+3},\Sp(n+1) \cdot \U(1)) = 
    \begin{dcases} 
    1, & n \text{  odd  }\\
    2 & n \text{  even  } \, .
 \end{dcases} 
\end{equation*}
Moreover, the $\Sp(2n+1) \cdot \U(1)$-invariant spin$^{\mathbb{C}}$ structures on $S^{8n+3}$ are the ones corresponding to 
\begin{align*}
    \varphi_s \colon \Sp(2n+1) \cdot \U(1) \to \U(1) \\
    [A,z] \mapsto z^s
\end{align*} 
with $s \equiv 2 \mod 4 $. 
\end{thmm}
\begin{proof}
Take the description of the isotropy representation $\sigma \colon \Sp(2n) \cdot \U(1) \to \SO(8n+3)$ given in~\cite{DKL}. Then, the Lie group homomorphism $ \Sp(2n) \cdot \U(1) \to \U(1)$ given by $[A,z] \mapsto z^2$ is well-defined, and, by looking at generators of the corresponding fundamental groups and applying Proposition~\ref{prop:fg}, one concludes. 

For the final assertion of the theorem, note that the only Lie group homomorphisms $\Sp(2n+1) \cdot \U(1) \to \SO(2) \cong \U(1)$ are $\varphi_s$, for $s \in \mathbb{Z}$ even. And $\sigma \times \varphi_s$ lifts to $\Spin^{\mathbb{C}}(8n+3)$ if and only if $s/2$ is odd, by Proposition~\ref{prop:fg}. 
\end{proof}

We also know that the $\Sp(n+1)\cdot \Sp(1)$-invariant spin type of $S^{4n+3}$ is $1$ for $n$ odd, and that it is not $1$ for $n$ even~\cite{DKL}. Now, we obtain: 

\begin{thmm}\label{thm:spheres_spsp}
    The $\Sp(2n+1) \cdot \Sp(1)$-invariant spin type of $S^{8n+3}$ is $3$. Hence, 
    \begin{equation*}
        \Sigma(S^{4n+3},\Sp(n+1) \cdot \Sp(1)) = 
        \begin{dcases} 
        1, & n \text{  odd  }\\
        3 & n \text{  even  } \, .
     \end{dcases} 
    \end{equation*}
    Moreover, there exists a unique $\Sp(2n+1) \cdot \Sp(1)$-invariant spin$^{\mathbb{H}}$ structure on $S^{8n+3}$ up to equivariant equivalence, namely the one corresponding to
\begin{align*}
    \varphi \colon \Sp(2n+1) \cdot \Sp(1) \to \SO(3) \\
    [A,q] \mapsto \lambda_3(q)
\end{align*} 
where $\lambda_3 \colon \Spin(3) \cong \Sp(1) \to \SO(3)$ is the standard double covering. 
    \end{thmm}
    \begin{proof}
    Take the description of the isotropy representation $\sigma \colon \Sp(2n) \cdot \Sp(1) \to \SO(8n+3)$ in~\cite{DKL}. Then, the Lie group homomorphism $ \Sp(2n) \cdot \Sp(1) \cong \Sp(2n) \cdot \Spin(3) \to \SO(3)$ given by $[A,z] \mapsto \lambda_3(z)$ is well-defined, and by looking at generators of the corresponding fundamental groups one concludes that the $\Sp(2n+1) \cdot \Sp(1)$-invariant spin type of $S^{8n+3}$ is either $2$ or $3$. But it cannot be $2$, because then we would have a non-trivial Lie group homomorphism $\SO(3) \to \SO(2)$, which is impossible, as in the proof of Proposition~\ref{prop:fg}.  

    The last assertion of the theorem follows from the fact that there is only one non-trivial $3$-dimensional representation of $\Sp(1)$ up to conjugation by elements of $\SO(3)$, together with Proposition~\ref{prop:fg} and Theorem~\ref{cor:inv_gen_spin}. 
    \end{proof}

We summarise our results in Table~\ref{table:invariant_spin_type}, which shows the $G$-invariant spin type of the spheres for all connected Lie groups $G$ acting transitively and effectively on them, and thus completing the study of invariant generalised spin$^r$ structures on homogeneous spheres.  

\begin{table}[h!]
    \centering
    \renewcommand{\arraystretch}{.8}
    \begin{tabular}{ c c c }
        \toprule
        Space $M$ & Group $G$ & $\Sigma(M,G)$ \\
        \midrule
        \multirow{2}*{$S^{n}$} & \multirow{2}*{$\SO(n+1)$} & $n$, for $n \neq 4$ \\ 
        & & $3$, for $n=4$ \\
        \midrule
        $S^{2n+1}$ & $\U(n+1)$ & $2$ \\ 
        \midrule
        $S^{2n+1}$ & $\SU(n+1)$ & $1$ \\ 
        \midrule
        $S^{4n+3}$ & $\Sp(n+1)$ & $1$ \\ 
        \midrule
        \multirow{2}*{$S^{4n+3}$} & \multirow{2}*{$\Sp(n+1) \cdot \U(1)$} & $1$, for $n$ odd \\ 
        & & $2$ , for $n$ even \\
        \midrule
        \multirow{2}*{$S^{4n+3}$} & \multirow{2}*{$\Sp(n+1) \cdot \Sp(1)$} & $1$, for $n$ odd \\ 
        & & $3$ , for $n$ even \\
        \midrule
        $S^{6}$ & $G_2$ & $1$ \\ 
        \midrule
        $S^{7}$ & $\Spin(7)$ & $1$ \\ 
        \midrule
        $S^{15}$ & $\Spin(9)$ & $1$ \\ 
        \bottomrule
    \end{tabular}
    \vspace{.2cm} 
    \caption{$G$-invariant spin type of homogeneous spheres.}
    \label{table:invariant_spin_type}
\end{table}  

Table~\ref{table:invariant_spin_type} has an interesting application, in analogy with~\cite[Prop.~1.6]{DKL}: 

\begin{thmm}\label{prop:general} 
    Let $G$ be the holonomy group of a simply connected irreducible non-symmetric Riemannian manifold of dimension $n+1 \geq 4$. Let $H \leq G$ be a subgroup such that $S^n \cong G/H$, which exists, by Berger's classification. Then, the following are equivalent: 
    \begin{enumerate} 
        \item There exists a homomorphic lift of the holonomy representation to $\Spin^{r}(n+1)$. 
        \item $S^n$ has a $G$-invariant spin$^r$ structure with strongly $G$-trivial auxiliary bundle. 
    \end{enumerate}
\end{thmm}

\begin{proof}
    Let $h \colon G \to \SO(n+1)$ be the holonomy representation, and let $\sigma \colon H \to \SO(n)$ be the isotropy representation of the corresponding sphere. We have a commutative diagram 
    \[ 
        \begin{tikzcd}
            H \arrow[d,"\sigma"] \arrow[r,"\iota"] & G \arrow[d,"h"] \\
            \SO(n) \arrow[r] & \SO(n+1) \, . 
        \end{tikzcd}
    \] 
    First, suppose that there exists a lift $\widetilde{h} \colon G \to \Spin^r(n+1)$ of $h$. Define $\psi \colon G \to \SO(r)$ by $\psi = \pr_2 \circ \varphi^{r,n} \circ \widetilde{h}$, and consider $\sigma \times \left(\psi \circ \iota \right) \colon H \to \SO(n) \times \SO(r)$. Then, we get the long exact sequences of homotopy groups 
        \[
\begin{tikzcd}[column sep=1em]
    \pi_2(G/H)=0 
        \arrow[r] 
    & \pi_1(H) 
        \arrow[d,"\sigma_{\sharp} \times \left( \psi \circ \iota \right)_{\sharp}"{right}] 
        \arrow[r,"\iota_{\sharp}"] 
    & \pi_1(G) 
        \arrow[r] 
        \arrow[d,"h_{\sharp} \times \psi_{\sharp}"{right}] 
    & \pi_1(G/H)=0 
    \\
    \pi_2(S^n)=0 
        \arrow[r] 
    & \pi_1\bigl(\SO(n)\times\SO(r)\bigr) 
        \arrow[r] 
    & \pi_1\bigl(\SO(n+1)\times\SO(r)\bigr) 
        \arrow[r] 
    & \pi_1(S^n)=0 \,.
\end{tikzcd}
\]
    This implies we have a commuting square where the horizontal maps are isomorphisms:
    \begin{equation}\label{eq:homotopy}
        \begin{tikzcd}
            \pi_1(H) \arrow[d,"\sigma_{\sharp} \times \left( \psi \circ \iota \right)_{\sharp}"] \arrow[r, "\cong"] & \pi_1(G) \arrow[d,"h_{\sharp} \times \psi_{\sharp}"] \\
            \pi_1(\SO(n) \times \SO(r)) \arrow[r, "\cong"] & \pi_1(\SO(n+1) \times \SO(r)) \, . 
        \end{tikzcd}
    \end{equation}
    To conclude, we need to show that the image of $\pi_1 \left( \Spin^r(n) \right)$ in $\pi_1(\SO(n) \times \SO(r))$ maps to the image of $\pi_1 \left(\Spin^r(n+1) \right)$ in $\pi_1(\SO(n+1) \times \SO(r))$ by the isomorphism of the bottom row. This follows by applying direct image to the commuting square:
    \[
        \begin{tikzcd}
            \Spin^r(n) \arrow[r] \arrow[d] & \Spin^r(n+1) \arrow[d] \\
            \SO(n) \times \SO(r) \arrow[r] & \SO(n+1) \times \SO(r) \, . 
          \end{tikzcd}
    \]
    This guarantees the existence of a lift of $\sigma \times (\psi \circ \iota)$ to $\Spin^r(n)$. By part (3) of Lemma~\ref{lemma:lift}, the auxiliary bundle of the corresponding spin$^r$ structure is strongly $G$-trivial. 

    Conversely, suppose that the sphere $G/H$ has a $G$-invariant spin$^r$ structure. This means that there exists a homomorphism $\varphi \colon H \to \SO(r)$ such that $\sigma \times \varphi$ lifts to $\Spin^r(n)$. If, moreover, the auxiliary bundle is strongly $G$-trivial, there exists a Lie group homomorphism $\psi$ extending $\varphi$ to $G$. Now, the same reasoning as above using diagram~\eqref{eq:homotopy} concludes the proof. 
\end{proof} 

\begin{rem}
    If we only require the auxiliary bundle to be trivial in point (2) of Theorem~\ref{prop:general} and not necessarily strongly $G$-trivial, the holonomy representation still lifts for topological reasons, but the lift need not be a homomorphism. 
\end{rem}

This result places the study of invariant spin$^r$ structures on spheres in a much wider context. For example, as a consequence of Theorem~\ref{prop:general} and Table~\ref{table:invariant_spin_type}, the holonomy representation of a simply connected irreducible non-symmetric $(8k+4)$-dimensional quaternionic K{\"a}hler manifold (or more generally any $(8k+4)$-dimensional manifold with holonomy containing $\Sp(2k+1)\cdot\Sp(1)$) does not lift to $\Spin^{\mathbb{C}}(8k+4)$.  
 
\FloatBarrier
\section*{Acknowledgements}
D. Artacho is funded by the UK Engineering and Physical Sciences Research Council (EPSRC), grant EP/W5238721. 


\bibliographystyle{alphaurl}
\bibliography{references.bib}

\begin{thebibliography}{DSKL22}

\bibitem[AC19]{AC}
D.~V. Alekseevsky and I.~Chrysikos.
\newblock Spin structures on compact homogeneous pseudo-{R}iemannian manifolds.
\newblock {\em Transform. Groups}, 24(3):659--689, 2019.
\newblock \href {https://doi.org/10.1007/s00031-018-9498-1} {\path{doi:10.1007/s00031-018-9498-1}}.

\bibitem[AH25]{AH24}
D.~Artacho and J.~Hofmann.
\newblock The geometry of generalised spin{$^r$} spinors on projective spaces.
\newblock {\em SIGMA Symmetry Integrability Geom. Methods Appl.}, 21:Paper No. 017, 32, 2025.
\newblock \href {https://doi.org/10.3842/SIGMA.2025.017} {\path{doi:10.3842/SIGMA.2025.017}}.

\bibitem[AHL23]{AHL}
I.~Agricola, J.~Hofmann, and M.-A. Lawn.
\newblock Invariant spinors on homogeneous spheres.
\newblock {\em Differential Geom. Appl.}, 89:Paper No. 102014, 2023.
\newblock \href {https://doi.org/10.1016/j.difgeo.2023.102014} {\path{doi:10.1016/j.difgeo.2023.102014}}.

\bibitem[AM21]{AM21}
M.~Albanese and A.~Milivojevi\'{c}.
\newblock {$\mathrm{Spin}^h$ and further generalisations of spin}.
\newblock {\em J. Geom. Phys.}, 164:Paper No. 104174, 13, 2021.
\newblock \href {https://doi.org/10.1016/j.geomphys.2021.104174} {\path{doi:10.1016/j.geomphys.2021.104174}}.

\bibitem[AM23]{AM21_2}
M.~Albanese and A.~Milivojevi\'{c}.
\newblock Corrigendum to ``{${\rm Spin}^h$} and further generalisations of spin'' [{J}. {G}eom. {P}hys. 164 (2021) 104174].
\newblock {\em J. Geom. Phys.}, 184:Paper No. 104709, 4, 2023.
\newblock \href {https://doi.org/10.1016/j.geomphys.2022.104709} {\path{doi:10.1016/j.geomphys.2022.104709}}.

\bibitem[B{\"a}r92]{Bar}
C.~B{\"a}r.
\newblock The {D}irac operator on homogeneous spaces and its spectrum on {$3$}-dimensional lens spaces.
\newblock {\em Arch. Math. (Basel)}, 59(1):65--79, 1992.
\newblock \href {https://doi.org/10.1007/BF01199016} {\path{doi:10.1007/BF01199016}}.

\bibitem[B{\"a}r99]{bar_spinh}
C.~B{\"a}r.
\newblock Elliptic symbols.
\newblock {\em Math. Nachr.}, 201:7--35, 1999.
\newblock \href {https://doi.org/10.1002/mana.1999900001} {\path{doi:10.1002/mana.1999900001}}.

\bibitem[CGT93]{CGT}
M.~Cahen, S.~Gutt, and A.~Trautman.
\newblock Spin structures on real projective quadrics.
\newblock {\em J. Geom. Phys.}, 10(2):127--154, 1993.
\newblock \href {https://doi.org/10.1016/0393-0440(93)90025-A} {\path{doi:10.1016/0393-0440(93)90025-A}}.

\bibitem[Che17]{thesis}
X.~Chen.
\newblock {\em Bundles of {I}rreducible {C}lifford {M}odules and the {E}xistence of {S}pin {S}tructures}.
\newblock ProQuest LLC, Ann Arbor, MI, 2017.
\newblock Thesis (Ph.D.)--State University of New York at Stony Brook.

\bibitem[Coh85]{cohen}
R.L. Cohen.
\newblock The immersion conjecture for differentiable manifolds.
\newblock {\em Ann. of Math. (2)}, 122(2):237--328, 1985.
\newblock \href {https://doi.org/10.2307/1971304} {\path{doi:10.2307/1971304}}.

\bibitem[{\v C}S09]{CSbook}
A.~{\v C}ap and J.~Slov{\'a}k.
\newblock {\em Parabolic geometries. {I}}, volume 154 of {\em Mathematical Surveys and Monographs}.
\newblock American Mathematical Society, Providence, RI, 2009.
\newblock Background and general theory.
\newblock \href {https://doi.org/10.1090/surv/154} {\path{doi:10.1090/surv/154}}.

\bibitem[DSKL22]{DKL}
J.~Daura~Serrano, M.~Kohn, and M.-A. Lawn.
\newblock {$G$}-invariant spin structures on spheres.
\newblock {\em Ann. Global Anal. Geom.}, 62(2):437--455, 2022.
\newblock \href {https://doi.org/10.1007/s10455-022-09855-z} {\path{doi:10.1007/s10455-022-09855-z}}.

\bibitem[EH16]{H1}
M.~Espinosa and R.~Herrera.
\newblock Spinorially twisted spin structures, {I}: {C}urvature identities and eigenvalue estimates.
\newblock {\em Differential Geom. Appl.}, 46:79--107, 2016.
\newblock \href {https://doi.org/10.1016/j.difgeo.2016.01.008} {\path{doi:10.1016/j.difgeo.2016.01.008}}.

\bibitem[Fri00]{friedrich}
T.~Friedrich.
\newblock {\em Dirac operators in {R}iemannian geometry}, volume~25 of {\em Graduate Studies in Mathematics}.
\newblock American Mathematical Society, Providence, RI, 2000.
\newblock Translated from the 1997 German original by Andreas Nestke.
\newblock \href {https://doi.org/10.1090/gsm/025} {\path{doi:10.1090/gsm/025}}.

\bibitem[FT00]{FT}
T.~Friedrich and A.~Trautman.
\newblock Spin spaces, {L}ipschitz groups, and spinor bundles.
\newblock volume~18, pages 221--240. 2000.
\newblock Special issue in memory of Alfred Gray (1939--1998).
\newblock \href {https://doi.org/10.1023/A:1006713405277} {\path{doi:10.1023/A:1006713405277}}.

\bibitem[HH58]{hopf}
F.~Hirzebruch and H.~Hopf.
\newblock Felder von {F}l\"{a}chenelementen in 4-dimensionalen {M}annigfaltigkeiten.
\newblock {\em Math. Ann.}, 136:156--172, 1958.
\newblock \href {https://doi.org/10.1007/BF01362296} {\path{doi:10.1007/BF01362296}}.

\bibitem[HH07]{spinq}
H.~Herrera and R.~Herrera.
\newblock {${\rm Spin}^q$} manifolds admitting parallel and {K}illing spinors.
\newblock {\em J. Geom. Phys.}, 57(7):1525--1539, 2007.
\newblock \href {https://doi.org/10.1016/j.geomphys.2007.01.002} {\path{doi:10.1016/j.geomphys.2007.01.002}}.

\bibitem[Hir76]{hirsch}
M.W. Hirsch.
\newblock {\em Differential topology}, volume No. 33 of {\em Graduate Texts in Mathematics}.
\newblock Springer-Verlag, New York-Heidelberg, 1976.
\newblock \href {https://doi.org/10.1007/978-1-4684-9449-5} {\path{doi:10.1007/978-1-4684-9449-5}}.

\bibitem[HNT18]{H3}
R.~Herrera, R.~Nakad, and I.~T\'{e}llez.
\newblock Spinorially twisted spin structures, {III}: {CR} structures.
\newblock {\em J. Geom. Anal.}, 28(4):3223--3277, 2018.
\newblock \href {https://doi.org/10.1007/s12220-017-9958-1} {\path{doi:10.1007/s12220-017-9958-1}}.

\bibitem[HS19]{H2}
R.~Herrera and N.~Santana.
\newblock Spinorially twisted {S}pin structures. {II}: twisted pure spinors, special {R}iemannian holonomy and {C}lifford monopoles.
\newblock {\em SIGMA Symmetry Integrability Geom. Methods Appl.}, 15:Paper No. 072, 48, 2019.
\newblock \href {https://doi.org/10.3842/SIGMA.2019.072} {\path{doi:10.3842/SIGMA.2019.072}}.

\bibitem[Law23]{lawson}
H.B. Lawson, Jr.
\newblock {${\rm Spin}^h$} manifolds.
\newblock {\em SIGMA Symmetry Integrability Geom. Methods Appl.}, 19:Paper No. 012, 7, 2023.
\newblock \href {https://doi.org/10.3842/SIGMA.2023.012} {\path{doi:10.3842/SIGMA.2023.012}}.

\bibitem[LS18]{CS1}
C.~I. Lazaroiu and C.~S. Shahbazi.
\newblock Complex {L}ipschitz structures and bundles of complex {C}lifford modules.
\newblock {\em Differential Geom. Appl.}, 61:147--169, 2018.
\newblock \href {https://doi.org/10.1016/j.difgeo.2018.08.006} {\path{doi:10.1016/j.difgeo.2018.08.006}}.

\bibitem[LS19]{CS2}
C.~I. Lazaroiu and C.~S. Shahbazi.
\newblock Real pinor bundles and real {L}ipschitz structures.
\newblock {\em Asian J. Math.}, 23(5):749--836, 2019.
\newblock \href {https://doi.org/10.4310/AJM.2019.v23.n5.a3} {\path{doi:10.4310/AJM.2019.v23.n5.a3}}.

\bibitem[LS22]{CS3}
C.~I. Lazaroiu and C.~S. Shahbazi.
\newblock Dirac operators on real spinor bundles of complex type.
\newblock {\em Differential Geom. Appl.}, 80:Paper No. 101849, 53, 2022.
\newblock \href {https://doi.org/10.1016/j.difgeo.2022.101849} {\path{doi:10.1016/j.difgeo.2022.101849}}.

\bibitem[Mor96]{Morgan}
John~W. Morgan.
\newblock {\em The {S}eiberg-{W}itten equations and applications to the topology of smooth four-manifolds}, volume~44 of {\em Mathematical Notes}.
\newblock Princeton University Press, Princeton, NJ, 1996.

\bibitem[Mor97]{moroianu}
A.~Moroianu.
\newblock Parallel and {K}illing spinors on {${\rm Spin}^c$} manifolds.
\newblock {\em Comm. Math. Phys.}, 187(2):417--427, 1997.
\newblock \href {https://doi.org/10.1007/s002200050142} {\path{doi:10.1007/s002200050142}}.

\bibitem[MS43]{class_spheres}
D.~Montgomery and H.~Samelson.
\newblock Transformation groups of spheres.
\newblock {\em Ann. of Math. (2)}, 44:454--470, 1943.
\newblock \href {https://doi.org/10.2307/1968975} {\path{doi:10.2307/1968975}}.

\bibitem[OT95]{quaternionic_monopoles}
C.~Okonek and A.~Teleman.
\newblock Quaternionic monopoles.
\newblock {\em C. R. Acad. Sci. Paris S\'{e}r. I Math.}, 321(5):601--606, 1995.
\newblock \href {https://doi.org/10.1007/BF02099718} {\path{doi:10.1007/BF02099718}}.

\bibitem[TV94]{Teichner}
P.~Teichner and E.~Vogt.
\newblock All 4-manifolds have spin$^c$ structures.
\newblock {\em unpublished note, available from the authors’ webpage}, 1994.
\newblock URL: \url{https://people.mpim-bonn.mpg.de/teichner/Math/ewExternalFiles/spin.pdf}.

\bibitem[Wan49]{Wang49}
H.-C. Wang.
\newblock Homogeneous spaces with non-vanishing {E}uler characteristic.
\newblock {\em Acad. Sinica Science Record}, 2:215--219, 1949.
\newblock \href {https://doi.org/10.2307/1969588} {\path{doi:10.2307/1969588}}.

\end{thebibliography}
\end{document}